\documentclass[12pt]{article} 
\usepackage{tikz} 
\usetikzlibrary{arrows.meta}
\usetikzlibrary{decorations.markings,
                ext.arrows
                }
\usepackage{mathtools}  
\usepackage{amsmath}
\usetikzlibrary{decorations.markings,arrows.meta,bending,hobby}
\usetikzlibrary{hobby}
\tikzset{number of arrows/.style={postaction={decorate,decoration = {markings,
mark=between positions 0 and {1-0.99/(#1)} step {1/(#1)} with {\arrow{Stealth[round, scale = 1.5]}}}}}}
\tikzset{
    arc arrow/.style args={
    to pos #1 with length #2}{
    decoration={
        markings,
         mark=at position 0 with {\pgfextra{
         \pgfmathsetmacro{\tmpArrowTime}{#2/(\pgfdecoratedpathlength)}
         \xdef\tmpArrowTime{\tmpArrowTime}}},
        mark=at position {#1-\tmpArrowTime} with {\coordinate(@1);},
        mark=at position {#1-2*\tmpArrowTime/3} with {\coordinate(@2);},
        mark=at position {#1-\tmpArrowTime/3} with {\coordinate(@3);},
        mark=at position {#1} with {\coordinate(@4);
        \draw[-{Latex[length=#2,bend]}]       
        (@1) .. controls (@2) and (@3) .. (@4);},
        },
     postaction=decorate,
     },
fixed arc arrow/.style={arc arrow=to pos #1 with length 2mm}     
}
\tikzset 
  {
   marking1reversed/.style=
     {decoration=
        {markings,
          mark=at position 0.5 with {\arrow[line width=0.7pt]{<}}  
        },
      postaction=decorate
     }
   }
\usetikzlibrary{hobby}
\tikzset{number of arrows/.style={postaction={decorate,decoration = {markings,
mark=between positions 0 and {1-0.99/(#1)} step {1/(#1)} with {\arrow{Stealth[round, scale = 1.5]}}}}}}
\usepackage[utf8]{inputenc}
\usepackage{hyperref}
\usepackage[T1]{fontenc}
\usepackage[a4paper, margin=2.7cm]{geometry}
\usepackage{amsmath, amsthm, amsfonts, amssymb}
\usepackage[nameinlink]{cleveref} 
\newcounter{dummy} \numberwithin{dummy}{section} 
\newtheorem{theorem}[dummy]{Theorem} 
\newtheorem{Definition}[dummy]{Definition}
\newtheorem{Corollary}[dummy]{Corollary}
\newtheorem{Proposition}[dummy]{Proposition}
\newtheorem{Lemma}[dummy]{Lemma}
\usepackage{authblk} 
\usepackage{blindtext,titlefoot} 
\title{\bf A Quaternionic Integration  \\  Similar to the Complex One}
\author{Michael Parfenov}
\affil{\small \textit{Bashkortostan Branch of Russian Academy of Engineering, Ufa, Russia}} 
 \date{}

\begin{document}
\maketitle\unmarkedfntext{\!\!\!\!\!\!\!\!\!\!\!2020 Mathematics Subject Classification: 30G35. Date:  Juni 19, 2025  \\e-mail address: mwparfenov@gmail.com as well as parfenm@gmx.de}
\begin{abstract}
 The conception of C- and H-representations of any  holomorphic function is further extended to the notions, definitions, lemmas and theorems of the complex integration. On this basis and the introduced notion of a H-plane, generalising the notion of a number complex plane, the theory of the quaternionic integration similar to the complex one  is built.   The complex   Taylor Theorem, Cauchy’s Theorem, Cauchy’s Integral Formula, Laurent's series, Laurent’s theorem, and Cauchy’s Residue Theorem are directly adapted to the quaternion case.         \newline   
\end{abstract}
\section{Introduction}
One of the problems of the quaternionic analysis  is to build a theory  similar to the complex one (see, for example \cite{sud:qua}, \cite{gs:nth}).
The theory of the so-called H-holomorphic functions solves this problem since the class of such quaternionic functions has the algebraic and differential properties fully identical to complex holomorphic analogues. The main principles and results of the quaternionic differentiation in this theory are briefly given in \cite{pm:onno}.

A quaternionic function is further denoted by  $$\psi\left(p\right)=\psi_{1}\left(x,y,z,u\right)+\psi_{2}\left(x,y,z,u\right)i+\psi_{3}\left(x,y,z,u\right)j+\psi_{4}\left(x,y,z,u\right)k,$$ where $x,y,z,u$ are real components of a quaternionic variable  $p=x+yi+zj+uk$; the functions $\psi_{1}\left(x,y,z,u\right),\psi_{2}\left(x,y,z,u\right),\psi_{3}\left(x,y,z,u\right),\psi_{4}\left(x,y,z,u\right)$ are real-valued and $i,j,k$ are the base quaternions of the quaternion space $\mathbb{H}$.
Using the Cayley–Dickson construction (doubling form)  we have
\begin{equation} \label{Eq1} 
 \begin{gathered}
p=a+b\cdot j\in \mathbb{H},\;\;\;\;\;\;\psi\left(p\right)=\Phi_{1}\left(a,b,\overline{a}, \overline{b} \right)+\Phi_{2\cdot}\left(a,b,\overline{a}, \overline{b} \right)\cdot j \in \mathbb{H},  \\
\,\,\, \overline{ p}=\overline{ a}-b\cdot j\subset \mathbb{H}\text{,} \;\;\;\;\;\; \overline{ \psi\left(p\right)}=\overline{\Phi_{1}\left(a,b,\overline{a}, \overline{b} \right)}\;- \Phi_{2\cdot}\left(a,b,\overline{a}, \overline{b} \right)\cdot j \subset \mathbb{H} \text{,}
\end{gathered}
\end{equation}
where
\begin{equation}  \label{Eq2}
 \begin{gathered}
 a=x+yi,\;\;\;\;\;b=z+ui, \\ 
 \overline{a}=x-yi,\;\;\;\;\; \overline{b}=z-ui,
\end{gathered}
\end{equation}
$$\Phi_{1}\left(a,b,\overline{a}, \overline{b} \right)=\psi_{1}+\psi_{2}i, \;\;\;\;\;\Phi_{2}\left(a,b,\overline{a}, \overline{b} \right)=\psi_{3}+\psi_{4}i,$$
$$\overline{ \Phi_{1}\left(a,b,\overline{a}, \overline{b} \right)}=\psi_{1}-\psi_{2}i,\;\;\;\; \overline{ \Phi_{2}\left(a,b,\overline{a}, \overline{b} \right)}=\psi_{3}-\psi_{4}i$$
are compex quantities, the \textquotedblleft  $\cdot$\textquotedblright  \,and overbar signs denote, respectively, quaternionic multiplication and complex (or quaternionic if needed) conjugation. For simplicity, we also use the short designations $\Phi_{1}$ and $\Phi_{2}$ respectively instead of $\Phi_{1}\left(a,b,\overline{a}, \overline{b} \right)$ and $\Phi_{2}\left(a,b,\overline{a}, \overline{b} \right)$.

 Let us recall some definitions, theorems and their consequences    \cite{pm:onno},\cite{pm:onde} necessary for further.
\begin{Definition} \label{de1.1}
It is assumed that the constituents $\Phi_{1}\left(a,b,\overline{a}, \overline{b} \right)$ and $\Phi_{2}\left(a,b,\overline{a}, \overline{b} \right)$ of a quaternionic function $\psi\left(p\right)=\Phi_{1}+\Phi_{2}\cdot j$ possess continuous first-order partial derivatives with respect to $a,\overline{a},b,$ and $\overline{b}$ in some open connected neighborhood $G_{4}\subset \mathbb{H}$ of a point $p\in \mathbb{H}$. Then a function $\psi\left(p\right)$ is said to be H-holomorphic (denoted by $\psi_{H}\!\left(p\right)$) at a point $p$ if and only if the functions $\Phi_{1}\left(a,b,\overline{a}, \overline{b} \right)$ and $\Phi_{2}\left(a,b,\overline{a}, \overline{b} \right)$ satisfy in $G_{4}$ the following quaternionic generalization of Cauchy-Riemann's equations:
\begin{equation} 
\label{Eq3} \left\{
\begin{aligned}
1)\,\,\,\,(\,\partial_{a}\Phi_{1}\!\!\mid \,\,=(\,\partial_{\overline{b}}\overline{\Phi_{2} }\! \mid, \,\,\,\,\,\,\,\,\,\,\,\,\,\,2)\,\,\,\,(\,\partial_{a}\Phi_{2}\!\!\mid &=-\,(\,\partial_{\overline{b}}\overline{\Phi_{1} }\! \mid 
,\\
3)\,\,\,\,(\,\partial_{a}\Phi_{1}\!\!\mid \,\,=(\,\partial_{b}\Phi_{2}\! \mid, \,\,\,\,\,\,\,\,\,\,\,\,\,\,4)\,\,\,\,(\,\partial_{\overline{a}}\Phi_{2}\!\!\mid &=-\,(\,\partial_{\overline{b}}\Phi_{1}\!\! \mid .
\end{aligned}
\right. 
\end{equation}
\end{Definition}
Here $\partial_{i},\,i=a,\overline{a},b,\overline{b},$ denotes the partial derivative with respect to $i$. The brackets $\left(\cdots{}\!\!\mid\right.$ with the closing vertical bar indicate that the transition $a=\overline{a}=x$ (to 3D space: $p=x+yi+zj+uk\rightarrow p_{3}=x+zj+uk$) has been already performed in expressions enclosed in brackets. Equations (\ref{Eq3}-1) and (\ref{Eq3}-2) have the components relating to the left quaternionic derivative  and equations (\ref{Eq3}-3) and (\ref{Eq3}-4) to the right one \cite[\!p.~\!\!4]{pm:onno}. The requirement  $a=\overline{a}=x$ provides a possibility of joint implementation of equations for the left and the right quaternionic derivatives.  Equations (\ref{Eq3}-1) and (\ref{Eq3}-3) as well as (\ref{Eq3}-2) and (\ref{Eq3}-4) become, respectively, identical after the transition to 3D space \cite[\!\!p.~\!\!4]{pm:onno}, i.e. the left derivative becomes equal to the right one.

We see that the H-holomorphy conditions (\ref{Eq3}) are defined so that during the check of the quaternionic holomorphy of any quaternionic function we have to do the transition  $a=\overline{a}=x$ in already calculated expressions for the partial derivatives of the functions $\Phi_{1}\left(a,b,\overline{a}, \overline{b} \right)$ and  $\Phi_{2}\left(a,b,\overline{a}, \overline{b} \right)$ and their complex conjugations.   Nevertheless, this doesn’t mean that we deal with triplets, since the transition $a=\overline{a}=x$ (or $y=0$) can not be initially done for quaternionic variables and functions. Otherwisewe we would lose a division operation. Any quaternionic function remains the same 4-dimensional quaternionic function regardless of whether we check its holomorphy or not. Simply put, the H-holomorphic functions are 4-dimensional quaternionic functions for which the partial derivatives of components of  the Cayley–Dickson doubling form satisfy equations (\ref{Eq3}) after the transition to 3D space. More clearly, they are those quaternionic functions whose the left and the right derivatives become equal after the transition to 3D space. 

One can construct H-holomorphic functions from their C-holomorphic analogues by using the following theorem.
\begin{theorem} \label{th1.2}
Let a complex-valued function $\psi_{C}\!\left(\xi\right):G_{2}\rightarrow \mathbb{C}$ of a complex variable $\xi$ be C-holomorphic everywhere in an open connected  set $G_{2}\subseteq \mathbb{C}$, except, possibly, \!at certain singularities. Then a H-holomorphic function $\psi_{H}\!\left(p\right)$ of the same kind as $\psi_{C}\!\left(\xi\right)$ can be constructed (without change of a kind of function) from $\psi_{C}\!\left(\xi\right)$ by replacing a complex variable $\xi\in G_{2}$ in an expression for $\psi_{C}\!\left(\xi\right)$ by a quaternionic variable $p\in G_{4}\subseteq \mathbb{H}$, where $G_{4}$ is defined (except, possibly, at certain singularities) by the relation $G_{4}\supset G_{2}$ in the sense that $G_{2}$  exactly follows from $G_{4}$ upon transition from $p$ to $\xi$.  
\end{theorem}

We also need the following theorem presented in \cite[,\!\!\!p.\!\!~7]{pm:onno}.
\begin{theorem} \label{th1.3}
The quaternionic generalization of the complex derivative has the following expression for the full (\textit{uniting the left and right derivatives}) quaternionic derivative  of the $k'\text{th}$ order: 
\begin{equation}
\label{Eq4} \psi_H^{\left(k\right)}\!\!\left(p\right)=\Phi_1^{(k)}+\Phi_2^{(k)}\cdot j, 
\end{equation}
where the constituents $\Phi_1^{(k)}$ and $\Phi_2^{(k)}$ are expressed by
\begin{equation} \label{Eq5}
\Phi_1^{(k)}=\partial_{a}\Phi_1^{(k-1)}\!\!+\partial_{\overline{a}}\Phi_1^{(k-1)} \,\,\,\text{and}\,\,\,\,\,\,\,\Phi_2^{(k)}=\partial_{a}\Phi_2^{(k-1)}\!\!+\partial_{\overline{a}}\Phi_2^{(k-1)};
\end{equation}
and
$\Phi_1^{(k-1)} \text{and}\,\, \Phi_2^{(k-1)}$ are the constituents of the $\left(k-1\right)^{'}\!\text{th}$ full derivative of  $\psi_{H}\left(p\right)$ represented in the Cayley–Dickson doubling form as $ \psi_H^{\left(k-1\right)}\!\!\left(p\right)=\Phi_1^{(k-1)}+\Phi_2^{(k-1)}\cdot j,  k\geq1;$ $\Phi_1^{(0)}=\Phi_{1}\!\!\left(a,b,\overline{a}, \overline{b} \right) \text{and}\,\, \Phi_2^{(0)}=\Phi_{2}\!\!\left(a,b,\overline{a}, \overline{b} \right).$

 \textit{If a quaternion function $\psi\left(p\right)$ is once H-differentiable in $G_{4}\subset \mathbb{H},$ then it possesses the full quaternionic derivatives of all orders in $G_{4}$ (except, possibly, at certain singularities), each one H-differentiable}. 
\end{theorem}

This theorem  leads  \cite{pm:onno} to the following   
\begin{Corollary} \label{co1.4} 
(The similarity between the differentiating rules for C-holomorhic and H-holomorhic functions).\! All expressions for derivatives of a H-holomorphic function $\psi_{H}\!\left(p\right)$ of the same kind as a C-holomorphic analogue $\psi_{C}\!\left(\xi\right)$  have the same forms as the expressions for corresponding derivatives of a function $\psi_{C}\!\left(\xi\right).$
\end{Corollary}

For example, if the second derivative of the C-holomorphic function $\psi_{C}\!\left(\xi\right)=\xi^{k},\,\,k=1,2,3,\dots,$ is $\psi_C^{\left(2\right)}\!\left(\xi\right)=  k\left(k-1\right)\xi^{k-2},$  then the second full derivative of the H-holomorphic function $\psi_{H}\!\left(p\right)=p^{k}$ is $\psi_H^{\left(2\right)}\!\left(p\right) =k\left(k-1\right)p^{k-2}.$

In accordance with \cite[,\!\!\!p.~\!6]{pm:onde} we can make the following
\begin{Proposition} \label{pr1.5}
 One can consider any holomorphic function  as one holistic notion, having two representations: a C-representation (in the case of C-holomorphic functions) and a  H-representation (in the case of H-holomorphic functions).  The transition from the C-representation to the H-representation is directly carried out by replacing a complex variable $\xi\in G_{2}\subseteq \mathbb{C}$ (as a single whole) in an expression for a C-holomorphic function $\psi_{C}\left(\xi\right)$ by a quaternionic variable $p\in G_{4}\subseteq \mathbb{H}$, where $G_{2}$ and $G_{4}$ are open domains defined just as in Theorem \ref{th1.2}. At that  the complex imaginary unit $i$ (if it is encountered) must be replaced by the quaternionic imaginary unit $I=\frac{yi+zj+uk}{\sqrt{y^{2}+z^{2}+u^{2}}}.$ 
\end{Proposition}

 The basis for this is as follows. Algebraic properties of the holomorphic functions in both representations are the same, including the fact that the quaternionic multiplication of H-holomorphic functions behaves as commutative and correspondingly the left quotient of  H-holomorphic functions is equal to the right one \cite[\!p.\!\!~8]{pm:onno}. The  differentiation  rules are also the same (Corollary \ref{co1.4}).  The rule for the complex multiplication follows from the rule for the quaternionic multiplication as well as the known complex Cauchy-Riemann equatians follow from their quaternionic generalization (\ref{Eq3}) upon the transition from the quaternion space to the complex plane \cite[,\!\!\!p.~\!\!6]{pm:onde}.
 We denote the quaternionic imaginary unit by $I$, as in \cite[,\!\!\!p.~\!\!280]{gs:nth};  this is more convenient here than the previously used (see, for example, \cite[p.p.~6,11]{pm:onno})  designation $r$.

 Since the  integration is the procedure inverse to the differentiation, and rules of the differentiation are identical for C- and H-holomorphic functions, the similarity of general definitions, theorems and formulas of quaternionic and complex integration of holomorphic functions should be expected. 
 
 \hspace{20mm}
 
 \textit{The aim of this article is to show that general principles, definitions, and theorems of a quaternionic integration of H-holomorphic functions  can be built  fully identical to complex ones.  In more detail, we want to show that all the results of the theory of complex integration can be rewritten as quaternion ones by replacing the complex variables with quaternion ones..} 
 
  \hspace{20mm}
 
As the aim implies, we will mostly repeat known complex definitions and proofs adapting them to the quaternionic case. To show such a possibility is one of tasks.  Within the framework of a short article  we cannot consider the theory of quaternionic integration in full; we only consider some problems. It seems there's not much novelty but a possibility  of such an adaptation to the quaternionic case should be considered at some point.
\section{Quaternionic Integration and Cauchy’s Theorem}
The introduced C- and H-representations of one holistic notion of a holomorphic function (Proposition \ref{pr1.5})  can be conveniently extended to paths and contour integrals when adapting all definitions and notions of the complex integration to the quaternionic area.
\subsection{Paths and contours in the quaternion space}

In the complex plane a path from $\xi_{0}$ to $\xi_{1}$ is defined as follows \cite[,\!\!\!p.~\!\!35]{st:ta}:  a path is a continuous function 
\begin{equation} \label{Eq6}
\gamma\left(t\right) :\left[\alpha,\beta\right]\rightarrow\mathbb{C},
\end{equation}
where $\left[\alpha,\beta\right]$ is a real interval. So, for each $\alpha\leq t\leq\beta$, $\gamma\left(t\right)$ is a point on the path. We say that the path $\gamma\left(t\right)$ starts at $\gamma\left(\alpha\right)$ and ends at $\gamma\left(\beta\right).$ We regard a path as a set of points in $\mathbb{C}$, i.e. we identify the function $\gamma\left(t\right)$ with its image, so that  $\gamma\left(\alpha\right)=\xi_{0}$ and $\gamma\left(\beta\right)=\xi_{1}$. However, we should
regard this set of points as having an orientation: a path starts at one end-point and ends at the other. The function $\gamma\left(t\right)$ is called a \textit{parametrization} of the path $\gamma$.
\textit{Withint the theory in question we call the function} $\gamma\left(t\right)$ \textit{a C-representation of a path} $\gamma.$

Replacing an independent complex variable $\xi$ by a quaternionic variable $p$ in (\ref{Eq6}) we define a H-representation of a path (i. e. a quaternionic path) as follows:
\begin{Definition}
A H-representation of a path from a quaternionic point $p_{0}$ to a quaternionic point $p_{1}$ is a function $\gamma\left(t\right)$ defined by
\begin{equation} \label{Eq7}
\gamma\left(t\right) :\left[\alpha,\beta\right]\rightarrow\mathbb{H},
\end{equation}
where $\left[\alpha,\beta\right]$ is a real interval. We say that the path $\gamma\left(t\right)$ starts at $\gamma\left(\alpha\right)$ and ends at $\gamma\left(\beta\right).$ We regard a path as a set of points in $\mathbb{H}$, i.e. we identify the function $\gamma\left(t\right)$ with its image, so that  $\gamma\left(\alpha\right)=p_{0}$ and $\gamma\left(\beta\right)=p_{1}$. We regard this set of points as having an orientation: a path starts at one end-point and ends at the other. We call the function $\gamma\left(t\right)$ a \textit{parametrization} of the path $\gamma$.
\end{Definition} 
\begin{Definition}
A path $\gamma :\left[\alpha,\beta\right]\rightarrow\mathbb{H}$ is called smooth if a derivative $\gamma^{'}$ exists and is continuous
throughout all of $\left[\alpha,\beta\right].$
\end{Definition} 
\begin{Definition}
Let $\gamma :\left[\alpha,\beta\right]\rightarrow\mathbb{H}$ be a path. If $\gamma \left(\alpha\right)=\gamma \left(\beta\right)$ (i.e. if $\gamma$ starts and ends at the same point) then we say that $\gamma$ is a closed path or a closed loop.
\end{Definition}
\begin{Definition}
A contour $\gamma$ is a collection of smooth paths $\gamma_{1},\gamma_{2},\dots ,\gamma_{n}$,  where the end-point of $\gamma_{r}$ coincides with the start point of $\gamma_{r+1}, 1\leq r\leq n-1.$ We write $$\gamma=\gamma_{1}+\gamma_{2}+\cdots +\gamma_{n}.$$ If the end-point of $\gamma_{n}$ coincides with the start point of $\gamma_{1}$ then we call $\gamma$ a closed contour.
\end{Definition}
\begin{Definition}
Let  $\gamma :\left[\alpha,\beta\right]\rightarrow\mathbb{H}$ be a path. Define  $-\, \gamma :\left[\alpha,\beta\right]\rightarrow\mathbb{H}$ to be the path $$-\,\gamma\left( t \right)=\gamma\left( \alpha +\beta -t \right).$$ We call $-\gamma$ the reversed path of $\gamma$.
\end{Definition}
Since expressions (\ref{Eq6}) and (\ref{Eq7}) (i.e. the expressions for complex and quaternionic parametrization) are similar, the all other definitions of paths in C- and H-representation are identical with the difference that in a C-representation we consider the paths $\gamma=\xi\left(t\right), t\in \left[\alpha,\beta\right]$ and in H-representation the paths $\gamma=p\left(t\right), t\in \left[\alpha,\beta\right].$ 
For simplicity we further consider smooth paths and open simply connected domains.
\begin{Definition} \label{def2.3}
Let $f:D\rightarrow \mathbb{H}$ be a continuous function on an open simply connected domain  $D\in\mathbb{H}$ ($D$ can be considered as $G_{4}$ in Definition \ref{de1.1} or Theorem \ref{th1.2}). We say that a function $F:D\rightarrow \mathbb{H}$ is an antiderivative of $f$ on $D$ if
\begin{equation} \label{Eq8}
F^{'}=f,
\end{equation} 
where the derivative of $F$ is defined by (\ref{Eq4}).
\end{Definition}
\begin{Definition}
We shall say that C- and H-representations have the same functional form if a functional dependence on a complex variable $\xi$ in a C-representation is the same as  a functional dependence on a quaternionic variable $p$ in a H-representation. We can obtain a H-representation from a C-representation for any expression by replacing a complex variable $\xi$ as a whole  by a quaternionic variable $p$, and the complex imaginary unit $i$ (if it is present) by the quaternionic imaginary unit $I$.
\end{Definition}
\begin{Proposition} \label{pr2.7}
(i)  The antiderivatives (if they exist) of one holomorphic function  in C- and H-representations have the same functional form. (ii) If an antiderivative in the C-representation exists and is C-holomorphic, then it exists and is H-holomorphic in the H-representation.
\end{Proposition}
\begin{proof}
 Let we have two holomorphic functions: $f\left(\xi\right)$ in C-representation  and $f\left(p\right)$ in H-representation.  Let they have the same functional form. Let us assume, to the contrary, that $f\left(p\right)$ have an antiderivative of the functional form different from an antiderivative of $f\left(\xi\right)$. Then, according (\ref{Eq8}) and Corollary \ref{co1.4}, we get the function $f\left(p\right)$ of the form different from  $f\left(\xi\right)$  (two functions with different dependences on variables cannot have the same derivatives), which is contrary to the initial statement that $f\left(\xi\right)$ has the same functional form as $f\left(p\right)$. This proves the statement (i).

(ii) Since the antiderivatives  of any function $f$ (if they exist) in C- and H-representations   have the same functional form, we can speak that a H-representation of an antiderivative of $f$ can be obtained from a C-representation of an antiderivative of $f$ by replacing a complex variable $\xi$ by a quaternionic $p$ in an expression for a C-representation, i.e. if a C-representation exists, then a H-representation also exists.  At that, according to Theorem \ref{th1.2}, if the antiderivative $F$ in a C-representation is C-holomorphic, then the antiderivative $F$ in a H-representation is H-holomorphic. This completes the proof of the statement (ii).
\end{proof}
Thus, a question of finding quaternionic antiderivatives of H-holomorphic functions can be boiled down to  a question of finding complex antiderivatives of C-holomorphic functions.  It follows (we can suppose) that the quaternionic antiderivative should not be sought if  there is no the complex analogue of the antiderivative.
\subsection{Quaternionic contour integration}
Similarly to complex analysis  \cite[,\!\!\!p.~\!\!122]{st:ta}, using the path parametrization (\ref{Eq7}), we introduce the following definition of a contour integral in the H-representation:
\begin{Definition} \label{def2.7}
Let $f:D\rightarrow \mathbb{H}$ be a continuous quaternionic function on an open simply  connected domain  $D\in\mathbb{H}$ (can be considered as $G_{4}$ in Definition \ref{de1.1} or Theorem \ref{th1.2}). Let $\gamma :\left[\alpha,\beta\right]\rightarrow D$ be a
smooth path in $D$. Then the integral of $f\left(p\right)$ along $\gamma$ is defined to be
\begin{equation} \label{Eq9}
\int_{\gamma}f\left(p\right)d p=\int_{\alpha}^{\beta}f\left(\gamma\left(t\right)\right)\gamma^{'}\left(t\right)dt. 
\end{equation}
\end{Definition}

Now we adapt the  Fundamental Theorem of Complex Contour Integration \cite[\!\!\!p.~\!\!134]{st:ta}  to the quaternion case.
\begin{theorem}\label{th2.10}(The quaternionic generalization of the fundamental theorem of contour integration) 
Suppose that $f:D\rightarrow \mathbb{H}$ is continuous on $D$, $ F:D\rightarrow \mathbb{H}$ is an antiderivative of $f\left(p\right)$ on $D$, and $\gamma$ is a contour from $p_{0}$ to $p_{1}.$ Then
\begin{equation} \label{Eq10}
\int_{\gamma}f\left(p\right) dp=F\left(p_{1}\right)-F\left(p_{0}\right).
\end{equation}
\end{theorem}
\begin{proof}
It is sufficient here to prove the theorem for smooth paths. Suppose that the path $\gamma:\left[\alpha,\beta\right]\rightarrow D,\,\gamma \left(\alpha\right)=p_{0},\gamma \left(\beta\right)=p_{1}$, is a smooth path. Let $\omega\left(t\right)=f\left(\gamma \left(t\right)\right)\gamma^{'}\left(t\right),$ where $\gamma^{'}\left(t\right)$ is a derivative of $\gamma \left(t\right)$ and let $W\left(t\right)=F\left(\gamma \left(t\right)\right).$ Then by the chain rule $$W^{'}\left(t\right)=F^{'}\left(\gamma \left(t\right)\right)\gamma^{'}\left(t\right)=f\left(\gamma \left(t\right)\right)\gamma^{'}\left(t\right)=\omega\left(t\right).$$ Write this expression in the the Cayley–Dickson doubling form (\ref{Eq1})  as follows: $$\omega\left(t\right)=\Phi_{1}\left(t\right)+\Phi_{2}\left(t\right)\cdot j\,\,\, \text{and}\,\,\ W\left(t\right)=U\left(t\right)+V\left(t\right)\cdot j$$ so that $U^{'}=\Phi_{1},\,\,\ V^{'}=\Phi_{2}.$ Hence, following (\ref{Eq9}), we get $$\int_{\gamma}f\left(p\right)dp=\int_{\alpha}^{\beta}f\left(\gamma\left(t\right)\right)\gamma^{'}\left(t\right)dt=\int_{\alpha}^{\beta}\omega\left(t\right) dt$$ $$=\int_{\alpha}^{\beta}\Phi_{1}\left(t\right) dt \,\,+\int_{\alpha}^{\beta}\Phi_{2}\left(t\right) dt\, \cdot j = U\left(t\right)\Bigr|_{\alpha}^{\beta}+V\left(t\right)\Bigr|_{\alpha}^{\beta}\cdot j \\=W\left(t\right)\Bigr|_{\alpha}^{\beta}=F\left(p_{1}\right)-F\left(p_{0}\right).$$ We have used the Fundamental Theorem of Calculus \cite[,\!\!\!p.~342]{st:ta} for the complex functions $U$ and $V$.
\end{proof} 
We see that integral (\ref{Eq10}) does not depend on a choice of a path $\gamma$ from $p_{0}$ to $p_{1}$; all we need to know is that there exists an antiderivative for $f\left(p\right)$ on a domain $D$ that contains $p_{0}, p_{1}.$ \\

\begin{Proposition} \label{pr2.11}
Expressions, including integrals of holomorphic functions and dependencies on variables as a whole, in a H-representation have the same form as those in a C-representation with a difference that in a C-representation they have a dependence on a variable $\xi$ and the complex imaginary unit $i$ (if it is encountered) and  in a H-representation the same dependence on $p$ and the quaternionic imaginary unir $I$.
\end{Proposition}
\begin{proof}
 This follows from the same initial determining formulas (\ref{Eq9}) and (\ref{Eq10}) in H-representation  and corresponding formulas in C-representation  \cite[,\!\!\!pp.~\!\!122,134]{st:ta}  differing from each other only by independent variables $\xi$ and  $p$ as well as the complex imaginary unit $i$ and the quaternionic imaginary unir $I$, since upon the transition from the C- to H-representation we also must \cite[\!\!p.\!\!~6]{pm:onno} replace the complex imaginary unit $i$ by  the quaternionic imaginary unir $I$.  In other words, the transition from the C- to the H-representation is carried out only by replacing the variables and imaginary units; anything more does not change.
\end{proof}

This statement we regard as sufficiently reasoned within the theory of H-holomorphic functions  because it follows from the full similarity of C- and H-representations. 
Cauchy’s Theorem states conditions under which $\oint_\gamma f=0$ when there is no initial reason for $f\left(p\right)$ to have an antiderivative \cite[,\!\!\!p.~\!\!169]{st:ta}. There is the fundamental proof of Cauchy’s Theorem for quaternionic functions  \cite[p.~20]{sud:qua}, following from Stoke's theorem (see also \cite[\!p.\!\!~5]{sh:ica}). It is complicated and we confine ourselves to a single simple proof sufficient for the aims of this article, using the advantage of the existence of C- and H-representations.
\begin{theorem} \label{th3.1} (Cauchy’s Theorem) 
Let $\psi_{H}\left(p\right)$ be H-holomorphic (i.e. without singularities) on a domain $D$ and let $\gamma$ be a closed contour in $D$. Then
 \begin{equation*}
\oint_\gamma \psi_{H}\left(p\right)dp =0.
\end{equation*}
\end{theorem}
\begin{proof}
Let Cauchy's Theorem in C-representation be fulfilled, i.e.  $\oint_\gamma \psi_{C}\left(\xi\right)d\xi =0.$ Suppose that this theorem in H-representation is not fulfilled, i.e. $\oint_\gamma\psi_{H}\left(p\right)dp\neq 0.$ Then, the transition from   H-representation to  C-representation by replacing a variable $p$ in it by a variable $\xi$ gives $\oint_\gamma \psi_{C}\left(\xi\right)d\xi \neq 0,$  which contradicts the initial statement that  $\oint_\gamma \psi_{C}\left(\xi\right)d\xi =0$.
\end{proof}
\hspace*{-0.78cm} \textbf{Remark} Notice that Cauchy's Theorem in H-representation follows from Theorem \ref{th2.10} at $p_{0}=p_{1}.$
The Cauchy Theorem also follows  from Proposition \ref{pr2.11}\,.
\section{Cauchy’s Integral Formula and Taylor’s Theorem}
\subsection{The quaternionic estimation lemmas}
\begin{Lemma} \label{le3.1} 
Let $ f_{1},f_{2}:\left[ \alpha,\beta \right]\longrightarrow \mathbb{C}$ be continuous functions. Then
 \begin{equation}
\left| \int_{\alpha}^{\beta} \left( f_{1}\left( t \right)+f_{2}\left(t\right)\cdot j \right)\, dt\right|\le \int_{\alpha}^{\beta}\left| f_{1}\left( t \right)+f_{2}\left(t\right)\cdot j \right|dt  
\end{equation}
\end{Lemma}
\begin{proof}
Write
\begin{equation*}
\int_{\alpha}^{\beta} \left( f_{1}\left( t \right)+f_{2}\left(t\right) \cdot j \right) dt =X+Y \cdot j,  
\end{equation*}
where $X,Y$ are some complex constants. \\
Then 
 \begin{align*}
&\left|X+Y\cdot j  \right|^{2}=\left| X \right|^{2}+\left| Y \right|^{2}=\overline{\left( X+Y\cdot j  \right)}\left( X+Y\cdot j  \right)=\left( \overline{X}-Y\cdot j  \right)\left( X+Y\cdot j  \right)\\
&=\left( \overline{X}-Y\cdot j  \right) \int_{\alpha}^{\beta} \left( f_{1}\left( t \right)+f_{2}\left(t\right) \cdot j \right) dt
=\int_{\alpha}^{\beta} \left( \overline{X}-Y\cdot j  \right) \left( f_{1}\left( t \right)+f_{2}\left(t\right) \cdot j \right) dt\\
&=\int_{\alpha}^{\beta} \left(\overline{X} f_{1}\left( t \right)+Y\overline{f_{2}\left(t\right)}  \right) dt +\int_{\alpha}^{\beta} \left(\overline{X} f_{2}\left( t \right)-Y\overline{f_{1}\left(t\right)}  \right) dt \cdot j.
\end{align*}

Since $\left|X+Y\cdot j  \right|^{2}\in \mathbb{R}$, the \textquotedblleft imaginary\textquotedblright \,second part of the above expression must be zero,
i.e. we have $\int_{\alpha}^{\beta} \left(\overline{X} f_{2}\left( t \right)-Y\overline{f_{1}\left(t\right)}  \right) dt \cdot j=0.$
 Then 
\begin{equation} \label{Eq12}
\left|X+Y\cdot j  \right|^{2}=\left| X \right|^{2}+\left| Y \right|^{2}=\int_{\alpha}^{\beta} \left(\overline{X} f_{1}\left( t \right)+Y\overline{f_{2}\left(t\right)}  \right) dt.
\end{equation} 

When considering quaternionic expressions in the Cayley–Dickson doubling form (\ref{Eq1})  we will call by analogy with complex analysis the part of the expression multiplied by the imaginary unit $j$ the \textquotedblleft imaginary \!\!\textquotedblright  \, part $Im$ of the expression. The remaining part of the expression without the unit $j$ can be called the \textquotedblleft real\!\! \textquotedblright \,part $Re$.
 Notice that the integrand in (\ref{Eq12}) is the \textquotedblleft real \!\!\textquotedblright  \, part of $ \left( \overline{X}-Y\cdot j  \right) \left( f_{1}\left( t \right)+f_{2}\left(t\right) \cdot j \right)$. Recall that for any complex number $\xi$ we have  $Re\left(  \xi \right)\le \left| \xi \right|.$ Then, given that $X,Y, f_{1}\left(t\right), f_{2}\left(t\right)$ are complex values, we obtain that
  \begin{align*}
 \overline{X} f_{1}\left( t \right)+Y\overline{f_{2}\left(t\right)}&\le \left| \left(  \overline{X}-Y \cdot j \right) \left( f_{1}\left( t \right) + f_{2}\left( t \right) \cdot j  \right)\right|\\
&=\left| \overline{X}-Y \cdot j \right|\left| f_{1}\left( t \right) + f_{2}\left( t \right) \cdot j \right|\\
&=\sqrt{\left| X \right|^{2}+\left| Y \right|^{2}} \left| f_{1}\left( t \right) + f_{2}\left( t \right) \cdot j \right|.
\end{align*}

Hence, using  (\ref{Eq12}), we have 
\begin{align*}
\left| X \right|^{2}+\left| Y \right|^{2}&=\int_{\alpha}^{\beta} \left(\overline{X} f_{1}\left( t \right)+Y\overline{f_{2}\left(t\right)}\right)dt\\
&\le\sqrt{\left| X \right|^{2}+\left| Y \right|^{2}} \int_{\alpha}^{\beta}\left| f_{1}\left( t \right) + f_{2}\left( t \right) \cdot j \right|dt 
\end{align*}
Cancelling the term $\sqrt{\left| X \right|^{2}+\left| Y \right|^{2}}$ gives $$\sqrt{\left| X \right|^{2}+\left| Y \right|^{2}}\le  \int_{\alpha}^{\beta}\left| f_{1}\left( t \right) + f_{2}\left( t \right) \cdot j \right|dt $$. Then we have $$\left| \int_{\alpha}^{\beta} \left( f_{1}\left( t \right)+f_{2}\left(t\right)\cdot j \right)\, dt\right|=\left| X+Y \cdot j \right|=\sqrt{\left| X \right|^{2}+\left| Y \right|^{2}}\le \int_{\alpha}^{\beta}\left| f_{1}\left( t \right) + f_{2}\left( t \right) \cdot j \right|dt $$ as claimed.
\end{proof}
\begin{Lemma} \label{le3.2}
(The estimation lemma) Let $f:D\longrightarrow  \mathbb{H}$ be continuous and let $\gamma$ be a contour in $D.$ Suppose that $\left| f\left( p \right) \right|\le M$ for all $p$ on $\gamma$. Then $$\left| \int_{\gamma}f\left(p\right) dp \right|\le M \, length\left( \gamma \right).$$
\end{Lemma}
\begin{proof}
Simply note that by Lemma \ref{le3.1} we have that
\begin{align*}
\left|\int_{\gamma}f\left(p\right)dp  \right|&=\left| \int_{\alpha}^{\beta}f\left(\gamma\left(t\right)\right)\gamma^{'}\left(t\right)dt  \right|  \\
&\le \int_{\alpha}^{\beta}\left| f\left(\gamma\left(t\right)\right)  \right|\left| \gamma^{'}\left(t\right) \right| dt \\
&\le M \int_{\alpha}^{\beta} \left| \gamma^{'}\left(t\right) \right| dt \\
&=M \, length\left( \gamma \right),   
\end{align*}
where the length of a contour $\gamma $ is denoted by $ length\left( \gamma \right)$.
\end{proof}
\subsection{H-representation of a number plane}
Recall that there are isomorphic forms of representating complex numbers and quaternions. The complex numbers are usually \cite{st:ta} represented as follows: 
\begin{equation}  \label{Eq13}
\xi=x+yi.
\end{equation} 
The quaternions $p=x+yi+zj+uk$ can be represented (see, for example, \cite[\!p.\!\!~281]{gs:nth}, also  \cite[\!p.\!\!~16]{pm:eac} ) in the following isomorphic (algebraically equivalent) to (\ref{Eq13}) form: 
\begin{equation}  \label{Eq14}
 p=x+vI,
\end{equation}
where $v=\sqrt{y^{2}+z^{2}+u^{2}}$ is a real value, $I=\frac{yi+zj+uk}{\sqrt{y^{2}+z^{2}+u^{2}}}$ is a purely imaginary unit quaternion, so its square is $-1$ and $\left| I \right|=1.$ Previously we used the designation $r$ for the quaternionic imaginary unit (see \cite[\!\!p.\!\!~5]{pm:onno},  \cite[\!\!p.\!\!~14]{pm:onde},  \cite[\!\!p.\!\!~16]{pm:eac},  \cite[\!\!p.\!\!~12]{pm:oncc}), however in this article the designation $I$ used in \cite[\!\!p.\!\!~280]{gs:nth} is most appropriate, since we want to reserve the $r$ for designations of radii of circles. We see that the complex imaginary unit $i$ is replaced (see also  \cite[\!\!p.\!\!~16]{pm:eac},  \cite[\!\!p.\!\!~14]{pm:onde}) by the quaternionic imaginary unit $I$  in addition to replacing the variable $\xi$ by $p$ (Theorem \ref{th1.2}) upon the transition from C-representation (\ref{Eq13}) to  H-representation (\ref{Eq14}). Such a replacement must be also made when transiting  in all complex expressions containing $i$.   

Based on the isomorphy of expressions (\ref{Eq13}) and (\ref{Eq14}) we can treat a set of points represented by  (\ref{Eq14}) as  \textquotedblleft \,a  complex plane with the imaginary unit $I$ \!\textquotedblright .
\,The   complex plane  with the imaginary unit $I$ is \textquotedblleft movable \textquotedblright upon changing coordinates  in the quaternionic space, however the main properties of the imaginary unit $I$: $I^{2}= -1$ and $\left| I \right|=1$  do not change.
 If we mentally will place each time  such a  movable  plane onto  \textquotedblleft \,a plane of a table \!\textquotedblright, then paths, contours, and moving along them will look like  those in an ordinary complex plane with the the imaginary unit $i$.  According to the concept of C- and H-representations we further will briefly call a plane with the imaginary unit  $I$ a H-plane and an ordinary complex plane  with the imaginary unit  $i$  a C-plane.   Correspondingly, we will also subdivide curves in H-plane into H-path, H-contour, H-circle, etc.  as well as in C-plane into C-path, C-contour, C-circle, etc.

In complex analysis we can use the following parametrization  \cite[\!\!p.\!\!~121]{st:ta} of a circle: $$S\left( \xi \right) =\xi_{0}+ re^{it}, 0\le t\le 2\pi,$$ where $\xi_{0}$ is centre and $r$ is radius of a circle.
 Using the specified rules of replacing we can introduce the following notion of a  simplest (primitive)  quaternionic generalization of a complex circle.
\begin{Definition}\label{asd}
We call a H-representation of a circle a H-circle and define it by
 \begin{equation} 
 S\left( p \right) =p_{0}+r e^{It}, 0\le t\le 2\pi,
\end{equation}
where $p_{0}$ is \textquotedblleft centre  \!\textquotedblright  of a H-circle and $r$ is its  \textquotedblleft radius\textquotedblright.  A H-circle starts at the point $t=0$ and ends at the point $t=2\pi$. The point travels along the H-circle in the anticlockwise direction (the positive orientation).
\end{Definition}

 We say that a closed quaternionic path $\gamma$ winds around a point $p_{0}$ if it winds around this point in the H-plane and we denote the winding number by  $\omega\left( \gamma, p_{0} \right)$ just as in complex plane \cite[\!p.\!\!~149]{st:ta}.  The different forms of closed paths are represented in Fig. \!1, including the closed path with the negative orientation in Fig 1a).

 \hspace{20mm}

\label{fig:numb1} \begin{tikzpicture}[closed  hobby]
 \node[fill,circle,label=below:$p_0$,inner sep=1.5pt] at (0,0){};
\node at (0,-1.6) {$\text{a)}$}; 
\draw[ marking1reversed] plot coordinates { (0:0.8) (90:1.2) (180:0.8) (270:1.1) }; 
\draw[->,ultra thin] (-1,-1.2)--(0.9,-1.2) node[right]{$\textit{ x}$};
\draw[->,ultra thin] (-1,-1.2)--(-1,1.3) node[above]{$\textit{vI}$}; 
\begin{scope}[xshift=3cm] 
 \node[fill,circle,label=below:$p_0$,inner sep=1.5pt] at (0,0){}; 
 \node at (0,-1.6) {$\text{b)}$};
 \draw[number of arrows=7] plot coordinates { (0:0.8) (90:1.2) (180:1) (270:1.1)
 (0:1.1) (90:0.9) (180:1.2) (270:0.7)};
 \draw[->,ultra thin] (-1.3,-1.2)--(0.9,-1.2) node[right]{$\textit{ x}$};
\draw[->,ultra thin] (-1.3,-1.2)--(-1.3,1.3) node[above]{$\textit{vI}$}; 
\end{scope}
\begin{scope}[xshift=6cm] 
 \node[fill,circle,label=below:$p_0$,inner sep=1.5pt] at (0,0){};
 \node at (0,-1.6) {$\text{c)}$}; 
 \draw[number of arrows=5] plot coordinates { (0:0.8) (90:1.2) (180:0.8) (270:1.1)
 (0:1) (45:1.3) (50:0.6)};
  \draw[->,ultra thin] (-0.9,-1.2)--(0.9,-1.2) node[right]{$\textit{ x}$};
\draw[->,ultra thin] (-0.9,-1.2)--(-0.9,1.3) node[above]{$\textit{vI}$}; 
\end{scope}
\begin{scope}[xshift=8.5cm] 
 \node[fill,circle,label=below:$p_0$,inner sep=1.5pt] at (0,0){};
 \node at (0,-1.6) {$\text{d)}$}; 
 \draw[number of arrows=4] plot coordinates { (-20:0.8) (45:1.2) (110:0.8) 
 (45:0.3)};
   \draw[->, thin] (-0.9,-1.2)--(0.9,-1.2) node[right]{$\textit{x}$}; 
\draw[->, thin] (-0.9,-1.2)--(-0.9,1.3) node[above]{$\textit{vI}$}; 
\end{scope}
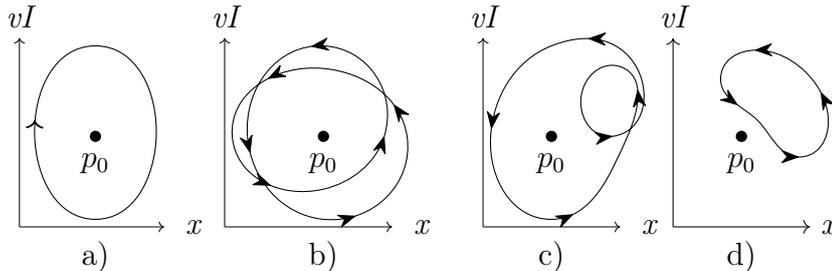
\hspace*{0.8cm} \node at (4.2,-2.3) (example-align) [draw=none, align=left]{Fig.\! \!1. Examples of paths  with different winding numbers: \\ a) $\omega\left( \gamma, p_{0} \right)=-1$,   b) $\omega\left( \gamma, p_{0} \right)=2$,\,\,c) $\omega\left( \gamma, p_{0} \right)=1$,\,\,d) $\omega\left( \gamma, p_{0} \right)=0.$}; \label{fig:numb1}
\end{tikzpicture}

 \hspace{20mm}

\hspace*{-0.78cm} \textbf{Remark} Notice that  we can draw the axis $vI$ perpendicularly to the axis $x$ on H-planes. According to \cite[\!p.\!\!~12]{pm:oncc} we have the polar form of a quaternion $p$ as follows: $$p=x+yi+zj+uk=m\left(\cos \theta+I\sin \theta\right),$$ where
$$m=\mid p \mid=\sqrt{x^{2}+y^{2}+z^{2}+u^{2}},$$
and an angle $\theta$ is defined from the expression
\begin{equation*} \tan \theta=\frac{v}{x}=\frac{\sqrt{y^{2}+z^{2}+u^{2}}}{x}.
\end{equation*} 
 In this formula $v=\sqrt{y^{2}+z^{2}+u^{2}}$ is an opposite (to an angle $\theta$) leg   of a right angled triangle, and $x$ is  an adjacent leg. These  are perpendicular to each other by the definition of the tangent. In complex analysis (C-representtion) we have the formula $\tan \theta=\frac{y}{x}$, where $y$ is an opposite leg and $x$ is perpendicular to it (also to $iy$)  an adjacent leg.  By analogy with this we can formally regard (draw) the axis $vI$ perpendicularly to the axis $x$  on the H-plane.

 \hspace*{-0.78cm} \textbf{Remark} We fix a position of a point $p_{0}$  in figures by using a single H-plane, on wich the transparent H-planes with a moving point are imposed.
 \begin{theorem}\label{th3.4}(The generalised version of Cauchy’s Theorem). 
Suppose that $\gamma_{1},\gamma_{2},\dots , \gamma_{n}$ are closed contours in a domain $D$ on a H-plane such that 
$$w\left(\gamma_{1},p\right)+w\left(\gamma_{2},p\right)+\cdots +w\left(\gamma_{n},p\right)=0,\,\,\, for\,\, all \,p\notin D$$
and let $f$ be H-differentiable in $D$. Then $$\int_{\gamma_{1}}f\left(p\right)dp +\int_{\gamma_{2}}f\left(p\right)dp+\cdots +\int_{\gamma_{n}}f\left(p\right)dp=0$$
\end{theorem}
\begin{proof}
Without going into too much detail note that since H-paths on the H-plane look like C-paths on the C-plane, the  choice of H-paths is identical to the choice in the complex case (see Figure 8.14 in \cite[\!\!p.\!\!~181]{st:ta}) and therefore the proof of this theorem is similar to complex one.
\end{proof}
\begin{theorem}\label{thgg}(Cauchy’s Integral Formula for a H-circle). 
Let $D \subset \mathbb{H}$ be a domain having a simple smooth boundary $C$, and the function $f\left(p\right)$ be  H-holomorphic in $D$. Then for any $p_{0}\in int D$ we have that
\begin{equation} \label{Eq16}
f\left( p_{0} \right)=\frac{1}{2\pi I}\int_{C} \frac{f\left( p \right)dp}{p-p_{0}}.
\end{equation}
\end{theorem}
\begin{proof}
We adapt  Theorem 2.1 from  \cite[\!\!p.\!\!~12]{lv:lcc} to the quaternion case. Consider the closed contours $C$ and $C_{\varepsilon}$ in $D$. Provided $\varepsilon >0$ is sufficiently small, we have that H-disc $D_{\varepsilon}=\left\{ p \in\mathbb{H} \mid \mid p-p_{0}  \mid \le   \varepsilon \right\}$ with a boundary $C_{\varepsilon}$   lies inside $D$. (In fact, we deal with a quaternionic ball, however using the notion of C- and H-representations we can speak of a H-disc and omit actual trajectories). Applying Cauchy's Theorem to a domain with a boundary $D\setminus \!\left\{int \left( D_{\varepsilon} \right)\right\}$ and the function $p\to \frac{f\left( p \right)}{p-p_{0}}, p  \neq p_{0}$ we obtain that
\begin{equation} \label{Eq17}
\frac{1}{2\pi I}\int_{C} \frac{f\left( p \right)dp}{p-p_{0}}=\frac{1}{2\pi I}\int_{C_{\varepsilon}} \frac{f\left( p \right)dp}{p-p_{0}}.
\end{equation}

 We see that it is sufficient to prove the equality
 \begin{equation} \label{Eq18}
f\left( p_{0} \right)=\frac{1}{2\pi I}\int_{C_{\varepsilon}} \frac{f\left( p \right)dp}{p-p_{0}}.
\end{equation}
Since the left quotient of one H-holomorphic function (here $f\left( p \right)$) by another H-holomorphic function (here $p-p_{0}$) equals the right one \cite[\!\!p.\!\!~20]{pm:eac}, we can use any of them in (\ref{Eq18}), for example the left quotient. 

Let $c$ be $c=f\left( p_{0} \right)$ and $f\left( p \right)$ be $f\left( p \right)=c+g\left( p \right)$,
where $g\left( p \right)$ is H-holomorphic and $g\left( p_{0} \right)=0$. Then the right-hand side of (\ref{Eq17}) can be written as follows: 
\begin{equation} \label{Eq19}
\frac{1}{2\pi I}\int_{C_{\varepsilon}} \frac{c\,dp}{p-p_{0}}+\frac{1}{2\pi I}\int_{C_{\varepsilon}} \frac{g\left( p \right)\,dp}{p-p_{0}}.
\end{equation}
We will show that the first term in this formula is equal to $c$ and the second term is equal to zero. 

To calculate the first integral we parameterize the boundary $C_{\varepsilon}$ of a H-disc  by the formula 
\begin{equation} 
C_{\varepsilon}: p=p_{0} +\varepsilon e^{2\pi It},\,\, t\in\left[0,1\right].
\end{equation}
Obviously, the orientation of the closed contour $C_{\varepsilon}$ with such a parametrization is positive. Further, since the value of $I$    is the same ($\mid \!I \!\mid =1$) at any point $p$, the value of $I$ may formally be regarded as a constant when differentiating with respect to $t$ and integrating round $C_{\varepsilon}$. Then $dp=2\pi I \varepsilon e^{2\pi I t}dt$ and we have
\begin{equation} \label{Eq21}
\int_{C_{\varepsilon}} \frac{c\,dp}{p-p_{0}}= \int_{0}^{1}\frac{c\,2\pi I \varepsilon e^{2\pi It}}{\varepsilon e^{2\pi It}}dt=2\pi Ic=2\pi I f\left( p_{0} \right).
\end{equation}
Thus the first integral in (\ref{Eq19}) is equal to $c$ and the first statement is true. 

The second integral in (\ref{Eq19})  can be estimated by using Lemmas  \ref{le3.1} and \ref{le3.2} as follows:
\begin{equation*} 
\left| \int_{C_{\varepsilon}} \frac{g\left( p \right)\,dp}{p-p_{0}} \right|\leq \int_{C_{\varepsilon}} \frac{\left| g\left( p \right)\right| \,dp}{ \left|p-p_{0} \right|}\leq\frac{M\left(\varepsilon\right)}{\varepsilon}2\pi \varepsilon=2\pi M\left(\varepsilon\right),
\end{equation*}
where $M\left(\varepsilon\right)=\sup_{\mid p-p_{0} \mid=\varepsilon}\mid \!g\left(p\right)\!\mid$. By virtue of Cauchy's Theorem the integrals in the last formula are independent of $\varepsilon$ (they are equal to integrals over $C$, see (\ref{Eq17})). On the other hand, since $g\left( p_{0} \right)=0$ and the function $g$ is continuous, we have $\lim_{\varepsilon \rightarrow 0} M\left(\varepsilon\right)=0$.
Thus the second integral in (\ref{Eq19}) is equal to zero and  the second statement is also true. This completes the proof of the theorem.
\end{proof}

The Cauchy integral formula allows us to express a H-differentiable function as a power series (the Taylor series expansion). Now we generalize some results from \cite{st:ta}.
\begin{theorem} \label{th3.6} (Quaternionic generalization of Taylor’s Theorem) 
Suppose that $f$ is H-holomorphic in a domain $D$ on the H-plane. Then all higher derivatives of $f\left(p\right)$ exist in $D$, and in any H-disc $$\left\{p\in\mathbb{H}\mid \mid p-p_{0}\mid <R \right\}\subset D$$
$f\left(p\right)$ has a Taylor series expansion given by $$ f\left(p\right)=\sum_{n=0}^{\infty} \frac{f^{\left(n\right)}\left(p_{0}\right)}{n!}\left(p-p_{0}\right)^{n}.$$
If \, $0<r <R$ and $ C_{r}\left(t\right)\!\!: p=p_{0}+r e^{It},\,0\leq t \leq 2\pi$, then $$f^{\left(n\right)}\left(p_{0}\right)=\frac{n!}{2\pi I}\int_{C_{r}}\frac{f\left(p\right)}{\left(p-p_{0}\right)^{n+1}}dp.$$
\end{theorem}
\begin{proof}
Given the similarity of C- and H-representations,  we can just as in complex analysis write for any $\mu \in \mathbb{H}$ the following expression: $$1+\mu +\mu^{2}+\cdots +\mu^{m}=\frac{1-\mu^{m+1}}{1-\mu}.$$
Put $\mu = \frac{h}{\left(p-p_{0}\right)},$ where $h$ is a fixed quaternionic quantity. Then $$1+\frac{h}{\left(p-p_{0}\right)}+\cdots +\frac{h^{m}}{\left(p-p_{0}\right)^{m}}=\frac{1-\left(\frac{h}{p-p_{0}}\right)^{m+1}}{1-\frac{h}{p-p_{0}}}=\frac{1-\left(\frac{h}{p-p_{0}}\right)^{m+1}}{p-p_{0}-h}\cdot\left(p-p_{0}\right).$$
All formulas hereinafter are correct, since for H-holomorphic functions the left and the right quotients are equal \cite[\!\!p.\!\!~20]{pm:eac} and we can use any of them. 

It follows that $$\frac{1}{\left(p-p_{0}\right)}+\frac{h}{\left(p-p_{0}\right)^{2}}+\cdots +\frac{h^{m}}{\left(p-p_{0}\right)^{m+1}}=\frac{1}{p-p_{0}-h}-\frac{h^{m+1}}{\left(p-p_{0}\right)^{m+1}\left(p-p_{0}-h\right)}.$$ Hence we have the following identity:
\begin{align*} 
\frac{1}{p-\left(p_{0}+h\right)}&=\frac{1}{p-p_{0}-h}=\frac{1}{\left(p-p_{0}\right)}+\frac{h}{\left(p-p_{0}\right)^{2}}+\cdots\\  
&+\frac{h^{m}}{\left(p-p_{0}\right)^{m+1}}+ \frac{h^{m+1}}{\left(p-p_{0}\right)^{m+1}\left(p-p_{0}-h\right)}.
\end{align*}

Fix $h$ such that $0<\mid \! h \!\mid <R$ and suppose that $\mid \! h \!\mid<r<R$. Then
Cauchy’s Integral formula, together with the above identity, gives
\[ \begin{split}
f\left(p_{0}+h\right)=&\frac{1}{2\pi I}\int_{C_{r}}\frac{f\left(p\right)}{p-\left(p_{0}+h\right)}dp       =\frac{1}{2\pi I}\int_{C_{r}}f\left(p\right)\left[\frac{1}{\left(p-p_{0}\right)}+\frac{h}{\left(p-p_{0}\right)^{2}}+\cdots \right. \\ &\left.+\frac{h^{m}}{\left(p-p_{0}\right)^{m+1}}+\frac{h^{m+1}}{\left(p-p_{0}\right)^{m+1}\left(p-p_{0}-h\right)}\right]dp=\sum_{n=0}^{m}a_{n}h^{n}+A_{m},
\end{split}  \]
where
\begin{equation} \label{Eq22}
a_{n}=\frac{1}{2\pi I}\int_{C_{r}}\frac{f\left(p\right)}{\left(p-p_{0}\right)^{n+1}}dp
\end{equation}
and $$A_{m}=\frac{1}{2\pi I}\int_{C_{r}}\frac{f\left(p\right)h^{m+1}}{\left(p-p_{0}\right)^{m+1}\left(p-p_{0}-h\right)}dp.$$
Now we will show that $A_{m}\rightarrow 0$ as $m\rightarrow\infty.$
 
 As $f$ is differentiable (H-holomorphic) on $C_{r},$ it is bounded. So there exists $M>0$ such that $\mid\! f\left(p\right) \!\mid \leq M$ for all $p$ on $C_{r}.$ By the reverse triangle inequality, using the fact that $\mid h \mid<r =\mid p-p_{0}\mid$ for $p$ on $C_{r}$, we
have that $$\mid p-p_{0}-h\mid\geq \mid \mid p-p_{0} \mid -\mid h \mid \mid=r\, - \mid h \mid\!.$$ 
Hence, by the Estimation Lemma (Lemma \ref{le3.2}) $$\mid A_{m} \mid\leq\frac{1}{2\pi}\frac{M {\mid h\mid}^{m+1}}{r^{m+1}\left(r\, -\mid h\mid\right)}2\pi r=\frac{M\mid h\mid}{r\, -\mid h\mid}\left(\frac{\mid h\mid}{r}\right)^{m}.$$ 
Since $\mid h \mid<r,$ this tends to zero as $m\rightarrow \infty.$ Hence $$f\left(p_{0}+h\right)=\sum_{n=0}^{\infty}a_{n}h^{n}$$ for $\mid h \mid<R$ with $$a_{n}=\frac{1}{2\pi I}\int_{C_{r}}\frac{f\left(p\right)}{\left(p-p_{0}\right)^{n+1}}dp,$$
provided that $r$ satisfies $\mid h \mid<r<R.$ However, the integral is unchanged if we vary $r$ in
the whole range $0<r<R$. Hence this formula is valid for the whole of this range of $r.$

Finally, we put $ h=p-p_{0}.$ Then we have that 
\begin{equation} \label{Eq23}
f\left(p\right)=\sum_{n=0}^{\infty}a_{n}{\left(p-p_{0}\right)}^{n}
\end{equation}
 for $\mid p-p_{0}\mid <R,$ with $a_{n}$ given as above. According to \cite[\!Proposition 3.9,\! \!p. \!10]{pm:onde} a quaternionic power series can be differentiated term-by-term as many times as we please and
\begin{equation} \label{Eq24}
a_{n}=\frac{f^{\left(n\right)}\left(p_{0}\right)}{{n!}}.
\end{equation}
Combining (\ref{Eq22}) and (\ref{Eq24}) we obtain as follows:$$f^{\left(n\right)}\left(p_{0}\right)=\frac{n!}{2\pi I}\int_{C_{r}}\frac{f\left(p\right)}{\left(p-p_{0}\right)^{n+1}}dp.$$ This completes the proof of the theorem.
\end{proof}
\subsection{Quaternionic imaginary unit as a H-holomorphic function}
To have  H-holomorphic terms $r_{k}p^{k}$ of power series  it was shown in  \cite[formula (24)]{pm:onde} that  the power series  coeffitients $r_{k},\,k=0,1,2,\dots,$ must be real-valued constants. There is no contradiction to formula (\ref{Eq22}), where Taylor's series coeffitients $a_{n}$  have the dependence on the quaternionic imaginary unit $I$ (generally can be written as $\eta\left(I\right)$).  We will show that the quaternionic imaginary unit $I$ regarded as a function of the variables $a,\overline{a},b,\overline{b}$ is H-holomorphic. Its compositions $\eta\left(I\right)$ in accordance with \cite{pm:oncc} are also H-holomorphic. Hence the terms $a_{n}h^{n}$ of Taylor's expansions are H-holomorphic as products of the H-holomorphic functions (coeffitients) $a_{n}$ and $h^{n}$.

Using (\ref{Eq2}) we obtain \cite[\!p.\!\!~13]{pm:oncc}  the quaternionic imaginary unit $I$ as follows: 
\begin{equation} 
I=\frac{yi+zj+uk}{\sqrt{y^{2}+z^{2}+u^{2}}}=I\left(a,b,\overline{a}, \overline{b} \right)=\Phi_{1}\left(a,b,\overline{a}, \overline{b} \right)+\Phi_{2\cdot}\left(a,b,\overline{a}, \overline{b} \right)\cdot j,
\end{equation}
where
\begin{align} 
&\Phi_{1}\left(a,b,\overline{a}, \overline{b} \right)=\frac{a-\overline{a}}{\sqrt{-\left(a+\overline{a}\right)^{2}+4\left(a\overline{a}+b\overline{b}\right)}},\\  
&\Phi_{2\cdot}\left(a,b,\overline{a}, \overline{b} \right)=\frac{2b}{\sqrt{-\left(a+\overline{a}\right)^{2}+4\left(a\overline{a}+b\overline{b}\right)}}.
\end{align}
Correspondingly, the conjugate of the function $I\left(a,b,\overline{a}, \overline{b} \right)$ is the following:
\begin{equation*} 
\overline{I\left(a,b,\overline{a}, \overline{b} \right)}=\Phi_{1}\left(a,b,\overline{a}, \overline{b} \right)+\Phi_{2\cdot}\left(a,b,\overline{a}, \overline{b} \right)\cdot j,
\end{equation*}
where
\begin{align*} 
&\Phi_{1}\left(a,b,\overline{a}, \overline{b} \right)=\frac{\overline{a}-a}{\sqrt{-\left(a+\overline{a}\right)^{2}+4\left(a\overline{a}+b\overline{b}\right)}},\\  
&\Phi_{2\cdot}\left(a,b,\overline{a}, \overline{b} \right)=-\frac{2b}{\sqrt{-\left(a+\overline{a}\right)^{2}+4\left(a\overline{a}+b\overline{b}\right)}}.
\end{align*}
To check the H-holomorphy of the function $\psi\left(p\right)=I\left(a,\overline{a},b,\overline{b}\right)$ we compute the necessary partial derivatives of the functions $\Phi_{1}\left(a,b,\overline{a}, \overline{b} \right)$ and $\Phi_{2}\left(a,b,\overline{a}, \overline{b} \right)\!:$
$\partial_{a}\Phi_{1}=\frac{4b\overline{b}}{\left(-a^{2}+2a\overline{a}-\overline{a}^{2}+4b\overline{b}\right)^{\frac{3}{2}}}, \, \partial_{\overline{b}}\overline{\Phi_{2} }=\partial_{b}\Phi_{2}=-\frac{2\left(a^{2}-2a\overline{a}+\overline{a}^{2}-2b\overline{b}\right)}{\left(-a^{2}+2a\overline{a}-\overline{a}^{2}+4b\overline{b}\right)^{\frac{3}{2}}},\, \partial_{\overline a}\Phi_{2}=-\partial_{\overline{b}}\overline{\Phi_{1}}=\frac{2\left(\overline{a}-a\right)b}{\left(-a^{2}+2a\overline{a}-\overline{a}^{2}+4b\overline{b}\right)^{\frac{3}{2}}}$, and  $\partial_{a}\Phi_{2}=-\partial_{\overline{b}}\Phi_{1}=\frac{2\left(a-\overline{a}\right)b}{\left(-a^{2}+2a\overline{a}-\overline{a}^{2}+4b\overline{b}\right)^{\frac{3}{2}}}.$ 

 After performing the transition $a=\overline{a}=x$ (to 3D space) we see that H-holomorphy equations (\ref{Eq3}) for the function $I=I\left(a,b,\overline{a}, \overline{b} \right)$ are fulfilled:
\begin{equation*}  \left\{
\begin{aligned}
1)\,\,\,(\,\partial_{a}\Phi_{1}\!\!\mid \,\,=(\,\partial_{\overline{b}}\overline{\Phi_{2} }\! \mid=\frac{1}{2\mid b \mid}, \,\,\,\,\,\,\,\,\,2)\,\,\,(\,\partial_{a}\Phi_{2}\!\!\mid &=-\,(\,\partial_{\overline{b}}\overline{\Phi_{1} }\! \mid=0
,\\
3)\,\,\,(\,\partial_{a}\Phi_{1}\!\!\mid \,\,=(\,\partial_{b}\Phi_{2}\! \mid=\frac{1}{2\mid b \mid}, \,\,\,\,\,\,\,\,\,4)\,\,\,(\,\partial_{\overline{a}}\Phi_{2}\!\!\mid &=-\,(\,\partial_{\overline{b}}\Phi_{1}\!\! \mid=0 .
\end{aligned}
\right. 
\end{equation*}
Thus, the quaternionic imaginary unit $I$ regarded as a function of variables $a,\overline{a},b,\overline{b}$   is H-holomorphic. 

Now we show that the H-holomorphic function $I\left(p\right)=I\left(a,b,\overline{a}, \overline{b} \right)$ does not change when changing the variable $p$, i.e. the first full derivative of this function with respect to $p$ is equal to zero.
Substituting the partial derivative $\partial_{\overline{a}}\Phi_{1}=-\frac{4b\overline{b}}{\left(-a^{2}+2a\overline{a}-\overline{a}^{2}+4b\overline{b}\right)^{\frac{3}{2}}}$ as well as the other needed partial derivatives calculated above into formulae (\ref{Eq4}) and (\ref{Eq5}) at $k=1$ we get the following expression for the first full quaternionic derivative of the function $I\left(p\right)$:
\begin{align*}
\left(I\left(p\right)\right)^{{\left(1\right)}} &=\Phi_1^{(1)}+\Phi_2^{(1)}\cdot j=\left(\partial_{a}\Phi_1\!+\partial_{\overline{a}}\Phi_1\right)+\left(\partial_{a}\Phi_2\!+\partial_{\overline{a}}\Phi_2\right)
 \cdot j \\ &=\left(\frac{4b\overline{b}}{\left(-a^{2}+2a\overline{a}-\overline{a}^{2}+4b\overline{b}\right)^{\frac{3}{2}}}-\frac{4b\overline{b}}{\left(-a^{2}+2a\overline{a}-\overline{a}^{2}+4b\overline{b}\right)^{\frac{3}{2}}}\right)\\&+\left(\frac{2\left(a-\overline{a}\right)b}{\left(-a^{2}+2a\overline{a}-\overline{a}^{2}+4b\overline{b}\right)^{\frac{3}{2}}}-\frac{2\left(a-\overline{a}\right)b}{\left(-a^{2}+2a\overline{a}-\overline{a}^{2}+4b\overline{b}\right)^{\frac{3}{2}}}\right)\cdot j\\&=0+0 \dot j=0.
\end{align*}
Thus, we confirm that the quaternionic imaginary unit $I$ can be  regarded as a constant that does not depend on $p$ (as noted above) when differentiating and integrating with respect to $p,$ correspondingly $t$.  Likewise we can speak of the H-holomorphic function of the conjugate quaternionic unit $\overline{I\left(a,b,\overline{a}, \overline{b} \right)}$.

\section{Laurent series and singularities} 
 The idea of Laurent series is to generalise (\ref{Eq23}) to allow negative powers of $\left(p-p_{0}\right)$. We adapt  the results of complex theory presented in \cite[\!\!p.\!\!~225]{st:ta} to the quaternion case.
\subsection{Laurent series}
\begin{Definition}
A quaternionic Laurent series is a series of the form
\begin{equation} \label{Eq28}
f\left(p\right)=\sum_{n=-\infty}^{\infty}a_{n}{\left(p-p_{0}\right)}^{n}.
\end{equation}
\end{Definition} Since (\ref{Eq28}) is a doubly infinite sum, we need to take care as to what it means.  Just as in complex analysis we define (\ref{Eq28}) to mean
\begin{equation*} 
\sum_{n=1}^{\infty}a_{-n}{\left(p-p_{0}\right)}^{-n}+\sum_{n=0}^{\infty}a_{n}{\left(p-p_{0}\right)}^{n})=\Sigma^{-}+\Sigma^{+}.
\end{equation*}

The first question to address is when does (\ref{Eq28}) converge? For this, we need both $\Sigma^{-}$ and $\Sigma^{+}$ to converge.
We can recognise $\Sigma^{-}$ as a power series in $\left(p-p_{0}\right)^{-1}$. This has a radius of convergence equal to, say, $R_1^{-1}\geq 0$, i.e. $\Sigma^{-}$ converges when $\mid \left(p-p_{0}\right)^{-1} \mid<R_1^{-1}$. In other words, $\Sigma^{-}$ converges when $\mid p-p_{0}\!\!\mid >R_{1}$.

Now $\Sigma^{+}$ converges for $\mid p-p_{0}\!\!\mid <R_{2}$ for some $ R_{2}>0$, where $R_{2}$ is the radius of convergence of $\Sigma^{+}$.

Combining these, we see that if $0\leq R_{1}<R_{2}\leq\infty$ then (\ref{Eq28}) converges on the H-annulus $$\left\{p\in\mathbb{H}\mid R_{1}<\mid p-p_{0} \mid <R_{2}\right\}.$$ 

We begin with an existence theorem for series expansions of the above kind. Given the similarity of the C- and H-representations, we can rewrite the known complex theorem \cite[\!\!p.\!\!~226]{st:ta} for the quaternion area as follows.
\begin{theorem} \label{th4.2} (Quaternionic generalization of Laurent’s theorem) 
If $f$ is differentiable (H-holomorphic) in the H-annulus $R_{1}\leq \mid p-p_{0} \mid \leq R_{2}$ where $0\leq R_{1}<R_{2}\leq\infty$, then
\begin{equation*} 
f\left(p_{0}+h\right)=\sum_{n=0}^{\infty}a_{n}{h}^{n}+\sum_{n=1}^{\infty}b_{n}{h}^{-n}
\end{equation*}
where $\sum_{n=0}^{\infty}a_{n}{h}^{n}$ converges for $\mid h \mid <R_{2}$, $\sum_{n=1}^{\infty}b_{n}{h}^{-n}$ converges for $\mid h \mid >R_{1}$, and in particular both sides converge in the interior of the H-annulus.

Further, if $C_{r}\left(t\right)=p_{0}+r e^{It}\!,\,t\in\left[0,2\pi\right]$ then
\begin{align*} 
&a_{n}=\frac{1}{2\pi I}  \int_{C_{r}}\frac{f\left(p\right)}{\left(p-p_{0}\right)^{n+1}} dp,\\  
&b_{n}=\frac{1}{2\pi I}\int_{C_{r}}f\left(p\right)\left(p-p_{0}\right)^{n-1}\,dp.
\end{align*}
\end{theorem}
\begin{proof}
We adapt the proof given in \cite[\!p.\!\!~226]{st:ta} to the quaternionic case. If $R_{1}<\mid h \mid <R_{2}$, choose $r_{1},\,r_{2}$ so that $$R_{1}<r_{1}<\mid h \mid <r_{2}<R_{2}$$ and let
\begin{align*} 
&C_{r_{1}}\left(t\right)=p_{0}+r_{1} e^{It},t\in\left[0,2\pi\right],\\  
&C_{r_{2}}\left(t\right)=p_{0}+r_{2} e^{It},t\in\left[0,2\pi\right]
\end{align*}
as shown in Fig.\! 2.

We show first that
\begin{equation} \label{Eq29}
f\left(p_{0}+h\right) =\frac{1}{2\pi I} \int_{C_{r_{1}}}\frac{f\left(p\right)dp}{\left(p-\left(p_{0}+h\right)\right)}-\frac{1}{2\pi I} \int_{C_{r_{2}}}\frac{f\left(p\right)dp}{\left(p-\left(p_{0}+h\right)\right)}. 
\end{equation}

Enclose $p_{0}+h$ in a small H-circle $$C_{\varepsilon}\left(t\right)=p_{0}+h+\varepsilon e^{It},t\in\left[0,2\pi\right].$$  

\begin{tikzpicture}[closed hobby, >=stealth] 
\centering \path [draw=none, fill=gray, even odd rule, fill opacity = 0.2] (0,-1) circle (2.8) (0,-1) circle (0.4); 
\draw[->,thin] (-3,-4)--(2.9,-4) node[right]{$\textit{x}$}; 
\draw[->,thin] (-3,-4)--(-3,1.7) node[above]{$\textit{vI}$}; 
\draw[fill=black](0,-1) circle (2 pt) node [black,yshift=0.16cm, xshift=-0.1cm, font=\scriptsize] {$p_{0}$};
   \draw [number of arrows=1] ((1.4,0.01) circle (0.55);  
 \draw[ fill=none] node [black,  yshift=-0.65cm,xshift=1.9cm] {$C_{\varepsilon}$}; 
\coordinate (p0+h) at (1.1:1); 
\fill (p0+h) (1.4,-0.06) circle (2pt);
 \node [above, black, yshift=-0.12cm, xshift=0.4cm,  font=\scriptsize] at (p0+h) {$p_0+h$}; 
\draw [number of arrows=1](0,-1) circle (0.7); 
\draw[fill=none] node [black,yshift=-2.0cm] {$C_{r_{1}}$};
\draw[fill=none, number of arrows=1](0,-1) circle (2.5) node [black,yshift=-1.5cm, xshift=1.5cm] {$C_{r_{2}}$}; 
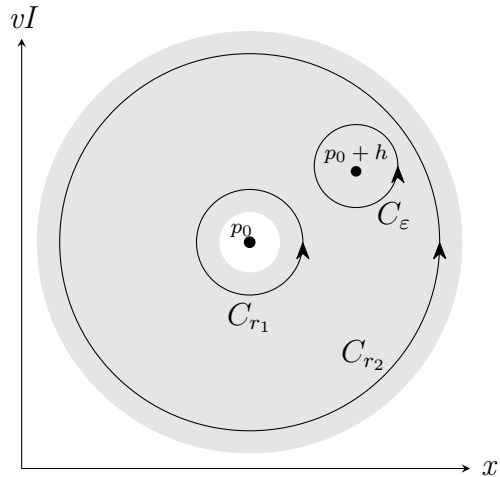
\node at (0,-4.5) {$\text{Fig. \!2. An annulus, showing paths used to prove Theorem 4.2.}$};
 \end{tikzpicture}
 
 \hspace{20mm}

Recall that we can  use  both the left  and the right quotients of $f\left(p\right)$ by $\left(p-\left(p_{0}+h\right)\right)$ in (\ref{Eq29}), since  they are equal in the case of H-holomorphic functions. 
\\Then the function $$\frac{f\left(p\right)}{\left(p-\left(p_{0}+h\right)\right)}$$
is H-differentiable in $S=\left\{p \in \mathbb{H}:R_{1}<\mid p-p_{0} \mid<R_{2},p\neq p_{0}+h\right\}$ and the contours $-C_{r_{2}}$ (with negative orientation), $\,C_{r_{1}},\,C_{\varepsilon}$ satisfy the hypotheses of Theorem  \ref{th3.4}: $\omega\left(C_{r_{1}},p_{0}\right)+\omega\left(-C_{r_{2}},p_{0}\right)+\omega\left(\,C_{\varepsilon},p_{0}\right)=1-1+0=0$ (see also Fig.1). Therefore $$\int_{C_{r_{1}}}f\left(p\right)dp\,+\int_{-C_{r_{2}}}f\left(p\right)dp\,+\int_{C_{\varepsilon}}f\left(p\right)dp=0.$$ Hence $$\int_{C_{\varepsilon}}f\left(p\right)dp=\int_{C_{r_{2}}}f\left(p\right)dp\,-\int_{C_{r_{1}}}f\left(p\right)dp.$$

However, by Cauchy’s Integral Formula $$\int_{C_{\varepsilon}}f\left(p\right)dp=2\pi I f\left(p_{0}+h\right)$$ giving (\ref{Eq29}). 

All we need to do now is work out the two integrals in (\ref{Eq29}) as power series, and calculate the coefficients. First, we choose $\rho_{1},\,\rho_{2}$ such that $$r_{1}<\rho_{1}< \mid h \mid<\rho_{2}<r_{2},$$ which enforces the conditions: $$(i)\,\mid p-\left(p_{0}+h\right) \mid>\rho_{1}-r_{1}\,\,\text{for}\,p\,\text{on}\,C_{r_{1}};$$ and $$(ii)\,\mid p-\left(p_{0}+h\right) \mid>r_{2}-\rho_{2}\,\,\text{for}\,p\,\text{on}\,C_{r_{2}}.$$
As in the proof of the Taylor series expansion (Theorem \ref{th3.6}) we get $$\frac{1}{2\pi I}\int_{C_{r_{1}}}\frac{f\left(p\right)}{p-\left(p_{0}+h\right)}dp=\sum_{n=0}^\infty a_{n} h^{n} $$ for $\mid h \mid<\rho_{2},$ where $$a_{n}=\frac{1}{2\pi I}\int_{C_{r_{1}}}\frac{f\left(p\right)}{\left(p-p_{0}\right)^{n+1}}dp.$$

Since $$\frac{1}{h}+\frac{p-p_{0}}{h^{2}}+\cdots+\frac{\left(p-p_{0}\right)^{n-1}}{h^{n}}-\frac{\left(p-p_{0}\right)^{n}}{h^{n}\left(p-p_{0}-h\right)}=\frac{-1}{p-p_{0}-h}$$ (summing a geometric series) we have
\[ \begin{split}
-&\frac{1}{2\pi I}\int_{C_{r_{2}}}\frac{f\left(p\right)}{p-p_{0}-h}dp  =\frac{1}{2\pi I}\int_{C_{r_{2}}}f\left(p\right) \left[  \frac{1}{h}+\frac{p-p_{0}}{h^{2}}+\cdots   \right. \\ &\left.   +\frac{\left(p-p_{0}\right)^{n-1}}{h^{n}}-\frac{\left(p-p_{0}\right)^{n}}{h^{n}\left(p-p_{0}-h\right)}\right]dp=\sum_{m=1}^{n}b_{m}h^{-m}-B_{n},
\end{split}  \]
where 
\begin{align*}
&b_{m}=\frac{1}{2\pi I}\int_{C_{r_{2}}}f\left(p\right)\left(p-p_{0}\right)^{m-1}dp\\
&B_{n}=\frac{1}{2\pi I}\int_{C_{r_{2}}}\frac{f\left(p\right)\left(p-p_{0}\right)^{n}}{h^{n}\left(p-p_{0}-h\right)}dp.
\end{align*}
Finally we estimate the size of $B_{n}.$ There exists $M>0$ such that $\mid f\left(p\right) \mid\leq M$ on $C_{r_{1}}.$ By (i) above, we have $\mid p-p_{0}-h \mid>\rho_{1}-r_{1}.$ Also $\mid h \mid >\rho_{1}, \,\,\mid p-p_{0} \mid=r_{1}.$ Hence, by Lemma \ref{le3.2}, $$\mid B_{n} \mid\leq\frac{1}{2\pi}\frac{Mr_1^n}{\rho_1^n \left(\rho_{1}-r_{1}\right)}2\pi r_{1}=\frac{Mr_{1}}{\rho_{1}-r_{1}}\left(\frac{r_{1}}{\rho_{1}}\right)^{n},$$
which tends to zero as $ n\rightarrow\infty,$ since $\frac{r_{1}}{\rho_{1}}< 1.$ It follows that $$-\frac{1}{2\pi I}\int_{C_{r_{2}}}\frac{f\left(p\right)}{\left(p-p_{0}-h\right)}dp=\sum_{m=1}^\infty b_{m} h^{-m}.$$

To finish the proof we must replace $C_{r_{1}}$ and $C_{r_{2}}$ by $C_{r}$ in the expressions for $a_{n}$ and $b_{n}.$ Since all three paths are homotopic (the paths are homotopic if one can be continuously deformed into the other) inside the annulus, this is
immediate.
\end{proof}
\subsection{Singularities and Poles}
Using the H-representation  of a plane,  i.e. the H-plane, it is not difficult to adapt the complex definitions of singularities and poles \cite[\!p.\!\!~230]{st:ta} ,  \cite[\!p.\!\!~232]{mh:ca} to the quaternion case.
\begin{Definition} 
A singularity of a function $f\left(p\right)$  is a point $p_{0}$ on a H-plane at which $f\left(p\right)$ is not H-differentiable, but every neighborhood  of $p_{0}$ contains at least one point at which $f\left(p\right)$ is H-differentiable.
\end{Definition}
\begin{Definition} 
If there exists a punctured H-disc $0<\mid p-p_{0} \mid <R$ such that  $f\left(p\right)$ is H-differentiable on this  disc then we say that $p_{0}$ is an isolated singularity of  $f\left(p\right)$.
\end{Definition}
Further, we will only be interested in isolated singularities. Assume that $f\left(p\right)$ has an isolated singularity at $p_{0}$ on a H-plane. Then $f\left(p\right)$  is H-holomorphic on an annulus of the form $\left\{p\in \mathbb{H}\mid 0<\mid p-p_{0} \mid <R \right\}$. We expand $f\left(p\right)$ as a Laurent series around $p_{0}$ on this annulus and obtain $$\sum_{n=0}^{\infty}a_{n}{\left(p-p_{0}\right)}^{n}+\sum_{n=1}^{\infty}b_{n}{\left(p-p_{0}\right)}^{-n}.$$ This is valid for $0<\mid p-p_{0} \mid <R.$ Consider the principal part of the Laurent series: 
 \begin{equation} \label{Eq30}
\sum_{n=1}^{\infty}b_{n}{\left(p-p_{0}\right)}^{-n}.
\end{equation}
There are three possibilities: the principal part of $f\left(p\right)$ may have

(i) no terms,

(ii) a finite number of terms,

(iii) an infinite number of terms.
\subsubsection{Removable singularities.}
Suppose that $f\left(p\right)$ has an isolated singularity at $p_{0}$ on a H-plane and that the principal part of the Laurent series  (\ref{Eq30}) has no terms in it. In this case, for $0<\mid p-p_{0} \mid <R$ we have that $$f\left(p\right)=a_{0}+a_{1}\left(p-p_{0}\right)+\cdots +a_{n}\left(p-p_{0}\right)^{n}+\cdots.$$ The radius of convergence of this power series is at least $R$, and so  $f\left(p\right)$ extends to a function
that is H-differentiable at $p_{0}$.\\
\textbf{Example}. Consider the function $$f\left(p\right)=\frac{\sin p}{p},\,\,\,p\neq 0.$$ Then $f\left(p\right)$ has an isolated singularity at $p=0$ as $f\left(p\right)$ is not defined at $p = 0.$ But we know \cite[\!p.\!\!~231]{st:ta} that $$\frac{\sin p}{p}=1-\frac{p^{2}}{3!}+\frac{p^{4}}{5!}-\cdots$$ for $p\neq 0.$ Define $f\left(0\right)=1.$ Then $f\left(p\right)$ is H-differentiable for all $p\in\mathbb{H}$. Hence  $f\left(p\right)$ has a removable singularity at $p=0.$
\subsubsection{Poles}
Suppose that $f\left(p\right)$ has an isolated singularity at $p_{0}$ and that the principal part of the Laurent series (\ref{Eq30}) has finitely many terms in it. In this case, for  $0<\mid p-p_{0} \mid <R$ we can write
$$f\left(p\right)=\frac{b_{m}}{\left(p-p_{0}\right)^{m}}+\cdots +\frac{b_{1}}{p-p_{0}}+\sum_{n=0}^{\infty}a_{n}\left(p-p_{0}\right)^{n},$$ where $b_{m}\neq 0.$ Then we say that $f\left(p\right)$ has a pole of order $m$ at $p_{0}$. A pole of order $1$ is called a simple pole.\\
\textbf{Example}. \label{ex1}
Let $$f\left(p\right)=\frac{\sin p}{p^{4}},\,\,\,p\neq 0.$$ This function  has an isolated singularity at $p=0$. We can write \cite[\!p.\!\!~231]{st:ta} 
$$\frac{\sin p}{p^{4}}=\frac{1}{p^{3}}-\frac{1}{3!}\frac{1}{p}+\frac{1}{5!}p\,-\frac{1}{7!}p^{3}+\cdots.$$ Hence $f\left(p\right)$ has a pole of order 3 at $p = 0$.

\subsubsection{Isolated essential singularities}
Suppose that $f\left(p\right)$ has an isolated singularity at $p_{0}$ and that the principal part of the Laurent series (\ref{Eq30}) has infinitely many terms in it. In this case we say that $f\left(p\right)$ has an isolated
essential singularity at  $p_{0}$.\\
\textbf{Example}.
Let $f\left(p\right)=\sin \frac{1}{p},\,\,p\neq 0$ (see complex analogue in \cite[\!p.\!\!~231]{st:ta}).  Then $f\left(p\right)$ has a singularity at $p=0$ and we have the series  $$\sin \frac{1}{p}=\frac{1}{p}-\frac{1}{3!}\frac{1}{p^{3}}+\frac{1}{5!}\frac{1}{p^{5}}-\cdots .$$ Hence  $f\left(p\right)$ has an isolated essential singularity at $p = 0$.

Recall that the main idea of this article is adaptation of every definition, theorem etc. of the complex integration to the quaternionic case by replacing a complex variable $\xi$ by a quaternionic $p$ and the imaginary unit $i$ by the quaternionic $I$. Therefore we have many repeatings known definitions, theorems and proofs from complex analysis on a new level of  quaternionic variables.

\section{Cauchy’s Residue Theorem} 
\subsection{Residues}
On the H-plane we can introduce the following definition.
\begin{Definition}
Suppose that $f\left(p\right)$ is H-holomorphic on a domain $D$ except for an isolated singularity at $ p_{0}\in D.$ Suppose that on 
$\left\{p\in \mathbb{H}\mid 0<\mid p-p_{0} \mid <R \right\}$, $f\left(p\right)$ has the Laurent expansion $$\sum_{n=0}^{\infty}a_{n}{\left(p-p_{0}\right)}^{n})+\sum_{n=1}^{\infty}b_{n}{\left(p-p_{0}\right)}^{-n}.$$
The residue of $f\left(p\right)$ at $p_{0}$ is defined to be $$Res\left(f,p_{0}\right)=b_{1}.$$ That is, the residue of $f\left(p\right)$ at the isolated singularity $p_{0}$ is the coefficient of $\left(p- p_{0}\right)^{-1}$ in the Laurent expansion.
\end{Definition}

Let $0<r<R.$ By Laurent’s Theorem (Theorem \ref{th4.2}) we have the alternative expression $$b_{1}=\frac{1}{2\pi I}\int_{C_{r}}f\left(p\right)\,dp,$$
where $C_{r}\left(t\right)=p_{0}+re^{It},\,\,t\in \left[0,2\pi\right]$ is a circular anticlockwise path around $ p_{0}$ in the annulus
of convergence. We further make the following definition.
\begin{Definition}
A closed contour $\gamma$ is said to be a simple closed loop if, for every point $p$ not on $\gamma$, the winding number is either $\omega\left( \gamma, p \right)=0$ or $\omega\left( \gamma, p \right)=1.$ If  $\omega\left( \gamma, p \right)=1$ then we say that $p$ is inside $\gamma$.
\end{Definition}
Thus a simple closed loop is a loop that goes round anticlockwise in a loop on H-plane once, and without intersecting itself.
\begin{theorem} \label{th5.3} (Cauchy’s Residue Theorem on a H-plane) 
Let $D$ be an open domain containing a simple closed loop $\gamma$ with the positive orientation and the points inside $\gamma$. Suppose that $f\left(p\right)$ is H-holomorphic on $D$ except for finitely many poles at $p_{1},p_{2}, \dots ,p_{n} $ inside $\gamma$. Then 
 \begin{equation*}
\int_{\gamma}f\left(p\right)dp=2\pi I\sum_{j=1}^n Res\left(f,p_{j}\right).
\end{equation*}
\end{theorem}
\begin{proof}
The proof is an  application of the Generalised Cauchy Theorem (Theorem \ref{th3.4}). 
Since $D$ is open, for each $ j=1,\dots ,n$, we can find nonintersecting circles $$S_{j}\left(t\right)=p_{j}+\varepsilon_{j}e^{It},\,t\in\left[0,2\pi\right]\!,$$ centred at $p_{j}$ and of radii $\varepsilon_{j},$ each described once anticlockwise, such that $S_{j}$ and the points inside  $S_{j}$ lie in $D$ and such that  $S_{j}$ contains no singularity other than $p_{j}$.

Let $D^{'}=D\setminus \left\{p_{1},\dots , p_{n}\right\}.$ We claim that the collection of paths $$-\gamma,S_{1},\dots ,S_{n}$$ satisfy the hypotheses of Theorem \ref{th3.4} with respect to $D^{'}:$ i.e. their winding numbers sum to zero for every point not in $D^{'}.$

To see this, first note that $$\omega\left(- \gamma, p \right)=\omega\left(S_{j}, p \right)=0\,\,\,\text{for}\,p \notin D.$$ Hence the hypotheses of the Generalised Cauchy Theorem hold for points $p$ not in $D.$ It
remains to consider points in $D$ that are not in $D^{'},$ i.e. the poles $p_{j}.$ 

Since each pole $p_{j}$ lies inside $\gamma$, we have that $$\omega \left(-\gamma,p_{j}\right)=-\omega \left(\gamma,p_{j}\right)=-1.$$ Moreover,
\begin{equation*} \omega \left(S_{k},p_{j}\right) =  \left\{
\begin{aligned}
0 \,\,\,\text{if}\,\,\,k\neq j\\
1\,\,\,\text{if}\,\,\,k= j
\end{aligned}
\right. 
\end{equation*}
and we have when considering each $p_{j}$ one circular anticlockwise contour $S_{j}$ for which  $p_{j}$ lies inside it and the other anticlockwise contours $S_{k}, k\neq j$, for which $p_{j}$ lies outside them (see, for example, a similar path in Fig.\! 1, d)), i.e. $\omega\left(S_{j}, p_{j}\right)=1$ and $\omega\left(S_{k}, p_{j}\right)=0$.

Hence $$\omega\left(- \gamma, p_{j} \right)+\omega\left(S_{1}, p_{j} \right)+\cdots +\omega\left(S_{n}, p_{j} \right)=0.$$
Then, by the Generalised Cauchy Theorem, $$\int_{-\gamma}f\left(p\right)dp +\int_{S_{1}}f\left(p\right)dp+\cdots +\int_{S_{n}}f\left(p\right)dp=0.$$
By Laurent’s Theorem we have that $$Res\left(f\left(p\right),p_{j}\right)= \frac{1}{2\pi I}\int_{S_{j}}f\left(p\right)\,dp.$$ 

Hence
\begin{align*}
\int_{\gamma}f\left(p\right)dp&=\int_{S_{1}}f\left(p\right)dp+\cdots +\int_{S_{n}}f\left(p\right)dp\\
&=2\pi I\left[Res \left(f\left(p\right),p_{1}\right)+\cdots +Res \left(f\left(p\right),p_{n}\right)\right].
\end{align*}
The proof of quaternionic Cauchy’s Residue Theorem is completed. 
\end{proof}

 \hspace{20mm}

We have shown that  the basic principles of  applying the concept of C- and  H-representations to integration expressions give the theory of the quaternionic integration similar to the complex one.  Extending the concept of C- and H-representations of holomorphic functions to all  expressions containing holomorphic functions can be used to obtain other results of a quaternionic analysis identical to the complex one.

  \label{sec:intro}
\bibliography{bibliography}
\end{document}